\documentclass{amsart}

\usepackage{amssymb}
\usepackage{amsrefs}
\usepackage{overpic}
\usepackage{enumitem}   
\usepackage{graphicx} 
\usepackage[hidelinks]{hyperref}
\usepackage{xcolor}

\graphicspath{{Figures/}} 

\newtheorem{theorem}{Theorem}
\newtheorem{proposition}{Proposition} 
\newtheorem{corollary}{Corollary} 
\newtheorem{remark}{Remark} 
\newtheorem{lemma}{Lemma}

\begin{document}\title[On a variant of Hilbert's 16th problem]{On a variant of Hilbert's 16th problem}
	
\author[Armengol Gasull and Paulo Santana]
{Armengol Gasull$^1$ and Paulo Santana$^2$}
	
\address{$^1$ Departament de Matem\`{a}tiques, Facultat de Ci\`{e}ncies, Universitat Aut\`{o}noma de Barcelona, 08193 Bellaterra, Barcelona, Spain ; and Centre de Recerca Matem\`{a}tica, Edifici Cc, Campus de Bellaterra, 08193 Cerdanyola del Vall\`{e}s (Barcelona), Spain}
\email{armengol.gasull@uab.cat}
	
\address{$^2$ IBILCE--UNESP, CEP 15054--000, S. J. Rio Preto, S\~ao Paulo, Brazil}
\email{paulo.santana@unesp.br}
	
	
	
\begin{abstract}
	We study the number of limit cycles that a planar polynomial vector field can have as a function of its number $m$ of monomials. We prove that the number of limit cycles increases at least quadratically with $m$ and we provide good lower bounds for $m\leqslant10$.
\end{abstract}
	
\maketitle
	
\section{Introduction and statement of the main results}\label{Sec1}

In his address to the International Congress of Mathematicians in Paris 1900, David Hilbert raised his famous list of problems for the $20$th century \cite{Browder}, with the $16$th problem being divided in two parts. In the first part motivated by \emph{Harnack's Curve Theorem} \cite{Harnack}, Hilbert asks from a description of the relative positions of the ovals of the algebraic curves satisfying Harnack's upper bound.

In the second part, motivated by finding an analogous to Harnack's result, Hilbert asks for the maximum number and relative position of limit cycles of planar polynomial vector fields. More precisely, given a planar polynomial vector field $X$, let $\pi(X)$ denote its number of limit cycles (i.e. isolated periodic orbits), where  the value infinity is also admitted. Let also $\mathcal{X}_n$ be the family of the planar polynomial vector fields $X=(P,Q)$ of degree $n$ (i.e. $\max\{\deg P,\deg Q\}=n$). The \emph{Hilbert number} $\mathcal{H}(n)\in\mathbb{Z}_{\geqslant0}\cup\{\infty\}$ is given by
	\[\mathcal{H}(n)=\sup\{\pi(X)\colon X\in\mathcal{X}_n\}.\]
The second part of Hilbert's $16$th problem consists in providing an upper bound for $\mathcal{H}(n)$, as a function of $n$, and a description of the relative position of such limit cycles. This problem is still open and is also part of Smale's list of problems for the $21$th century \cite{Smale}. In his own words: \emph{except for the Riemann hypothesis it seems to be the most elusive of Hilbert’s problems}. Despite the many attempts, no progress was made in finding upper bounds for $\mathcal{H}(n)$. So far it is not even known if $\mathcal{H}(2)$ is finite or not. While it has not been possible to find upper bounds for $\mathcal{H}(n)$, there has been success in obtaining lower bounds. It is known that $\mathcal{H}(n)$ increases at least as fast as $O(n^2\ln n)$. See \cite{ChrLlo1995,HanLi2012}. In fact, it was even conjectured in 1988 by Lloyd that $\mathcal{H}(n)$ is of order $O(n^3),$ see \cite{Lloyd}. For lower values of $n$, as far as we know, at this moment the best lower bounds are $\mathcal{H}(2)\geqslant 4$ \cites{ChenWang1979,Son1980}, $\mathcal{H}(3)\geqslant 13$ \cite{LiLiuYang2009} and $\mathcal{H}(4)\geqslant 28$ \cite{ProTor2019}. For more lower bounds, we refer to \cites{ProTor2019,HanLi2012}. 

In this paper we study a variant of Hilbert's $16$th problem. Instead of looking at the number of limit cycles as a function of the degree of $X$, we look it as a function of the number of monomials. 

We now provide a precise statements of our main results. Given a planar polynomial vector field $X=(P,Q)$, we say that $X$ has $m$ monomials if the sum of the number of monomials of $P$ and $Q$ is equal to $m$. Let $\mathcal{M}_m$ be the family of planar polynomial vector fields with $m$ monomials, independently of its degree. We define the \emph{Hilbert monomial number} $\mathcal{H}^M(m)\in\mathbb{Z}_{\geqslant0}\cup\{\infty\}$ as
	\[\mathcal{H}^M(m)=\sup\{\pi(X)\colon X\in\mathcal{M}_m\}.\]
So far very little is known about $\mathcal{H}^M(m)$. It follows from Buzzi et al \cite{BuzGasYag2021} that $\mathcal{H}^M(m)=0$ for $m\in\{1,2,3\}$, $\mathcal{H}^M(m)\geqslant m-3$ for $m\geqslant4$ and that there is a sequence of positive integer numbers $m_k\to\infty$, such that $\mathcal{H}^M(m_k)\geqslant N(m_k)$, with $N(m)$ of order $O(m\ln m)$. This second lower bound follows from the results of Álvarez and collaborators \cite{ACDP} obtained for Liénard type vector fields and it can  be seen that it can also be obtained from the lower bound of type $O(n^2\ln n)$ of $\mathcal{H}(n).$   

In our first main result we improve these general lower bounds proving that $\mathcal{H}^M(m)$ increases at least with order $O(m^2)$. 

\begin{theorem}\label{Main1}
	If $m\geqslant 9$, then $\mathcal{H}^M(m)\geqslant\frac{1}{2}m^2-3m-8$.
\end{theorem}

As we will see, our proof is based on the study of some Abelian integrals and it is  self-contained.

We remark that the main goal of the above result is only to show the quadratic growth of  $\mathcal{H}^M(m)$. For small $m$ the given lower bound is not good at all. For instance the result shows that $\mathcal{H}^M(10)\geqslant12$ while in our next result we prove that  $\mathcal{H}^M(10)\geqslant32$. In fact, as we will see, Theorem~\ref{Main1} is a corollary of the sharper result given in Proposition~\ref{Lemma2}: {\it for any non-negative integer numbers $n$ and $r$, there are planar polynomial vector fields with $n+r+4$ monomials and at least $2n(r+1)+n\big(1+(-1)^r\big)$ limit cycles.} 
Next we will study in more detail better lower bounds of $\mathcal{H}^M(m)$ for $m\le10.$

It follows among the series of papers about the limit cycles of cubic Li\'enard systems of Dumortier and Li that $\mathcal{H}^M(6)\geqslant4$ \cite{DL2} and $\mathcal{H}^M(7)\geqslant5$ \cite{DL4}. Also, it follows from Chow et al \cite[Sect. $4.2$]{ChoLiWan1994} that $\mathcal{H}^M(5)\geqslant 3$. In recent years Bréhard et al \cite[Chap. $6$]{Brehard1} and \cite[Sect. $7$]{Brehard2} developed a computed assisted method to study the zeros of Abelian integrals. With this method they provided a computed assisted proof of the existence of a quartic vector field with at least $24$ limit cycles. Since this vector field has only nine momonials, it follows that $\mathcal{H}^M(9)\geqslant24$. As far as we known, these are the only specific lower bounds known for small values of $m$. In our second main result we obtain better lower bounds for values of $4\leqslant m \leqslant 10.$ For $m=9$,  we replicate the known lower bound $\mathcal{H}^M(9)\geqslant24$ with a direct proof. For a summary of the previous and new lower bounds, see Table~\ref{Table1}.

\begin{theorem}\label{Main2} If $m\in\{4,5,6\}$, then 
	$\mathcal{H}^M(m)\geqslant12$. Moreover, $\mathcal{H}^M(7)\geqslant16$, $\mathcal{H}^M(8)\geqslant20$, $\mathcal{H}^M(9)\geqslant24$, and $\mathcal{H}^M(10)\geqslant32$.
\end{theorem}

To illustrate some of the vector fields involved in the proof of the above theorem we show the two families of vector fields that we have used to prove that $\mathcal{H}^M(4)\geqslant12$ and $\mathcal{H}^M(9)\geqslant24$. They are 
\begin{equation}\label{example1}
		\dot x=\alpha_1 x^{g_{11}}y^{g_{12}}-\beta_1 x^{h_{11}}y^{h_{12}},\quad
		\dot y=\alpha_2 x^{g_{21}}y^{g_{22}}-\beta_2 x^{h_{21}}y^{h_{22}},
\end{equation}
for some $\alpha_i$, $\beta_i\in\mathbb{R}$ and $g_{ij}$, $h_{ij}\in\mathbb{Z}_{>0}$, and 
\begin{equation}\label{example2}
	\dot x=y-y^3+\sum_{k=0}^{5}(-1)^kx^{2(5-k)+1}\left(\frac{y}{a_k}\right)^{2m_k}, \quad \dot y=x,
\end{equation}
with $a_i>0$, $m_i\in\mathbb{Z}_{>0}$ and $1\ll m_1\ll\dots\ll m_5$, respectively. Notice that they have respectively 4 and~9 monomials, and we will show that there are values of the parameters with at least 12 and 24 limit cycles, respectively. The first one~\eqref{example1} is constructed from  a so called \emph{Planar-S system} studied at \cite{BorHofDCDS,BorHofJDE} and having three limit cycles. That planar-S system is exactly of the form~\eqref{example1}, but with exponents $g_{ij}$, $h_{ij}\in\mathbb{R},$ and it is only defined in the first quadrant. As we will see, by perturbing these exponents (to transform them into rational numbers) and after some suitable changes of variables and time we will arrive to a new system of the form~\eqref{example1} that has at last three limit cycles  in each quadrant, providing the desired lower bound. The second one \eqref{example2} is studied by using Abelian integrals.

\begin{table}[h]
	\caption{Summary of the lower bounds of the Hilbert monomial numbers. Recall that $\mathcal{H}^M(m)=0$ for $m\leqslant3$.}\label{Table1}
	\begin{tabular}{c c c} 
		\hline
		\text{Monomials} & \text{New lower bounds} & \text{Previous lower bounds} \\
		\hline
		$4$ & $12$ & $1$ \\
		$5$ & $12$ & $3$ \\
		$6$ & $12$ & $4$ \\
		$7$ & $16$ & $5$ \\
		$8$ & $20$ & $5$ \\
		$9$ & $24$ & $24$ \\
		$10$ & $32$ & $24$ \\
		$m\geqslant11$ & $\frac{1}{2}m^2-3m-8$ & $m-3$ \\
		Asymptotic & $O(m^2)$ & $O(m\ln m)$ \\
		\hline
	\end{tabular}
\end{table}

We remark that in the third column of Table~\ref{Table1}, the lower bound $\mathcal{H}^M(10)\geqslant24$ follows from the fact that in the previous known lower bound $\mathcal{H}^M(9)\geqslant24$, all the limit cycles have odd multiplicity and thus are persistent under small perturbations. Similarly the two lower bounds in the second column for $m=5$ and $6$ follow from the one obtained from $m=4.$    It is natural to believe that these two lower bounds could be  improved, but until now, we have not been able to do it. 

It is curious to observe that if we address to a similar question but for planar polynomial vector fields written in complex coordinates, that is the ones given by $\dot z=F(z,\bar z)$, where $F$ is a polynomial with $m$ monomials, a totally different result happens. On the one hand, these vector fields with $m=1$ or $m=2$ have at most 0, or 1 limit cycle, respectively~\cite{AGP}. On the other hand, when $m=3$ (or higher) there is no upper bound for the number of limit cycles~\cite{GLT}. 

The idea of looking for the number of monomials instead of the degree of polynomials goes back to Descartes and his \emph{rule of signs}, which states that if $p\colon\mathbb{R}\to\mathbb{R}$ is a polynomial with $m$ nonzero monomials, independently of its degree, then $p$ has at most $m-1$ positive real roots, counting with multiplicity. In particular, it also follows that $p$ has at most $2m-1$ distinct real roots ($m-1$ positive, $m-1$ negative and eventually the root $x=0$, which can be of any multiplicity). Moreover, there are attempts to extended Descartes' rule of signs to the multiple variable case, such as the Kouchnirenko’s conjecture (nowadays known to be false). For more details, we refer to Problems $28$ and $29$ of \cite{Gas2021}.

Furthermore in more recent developments on real algebraic geometry, Harnack's Curve Theorem is replaced by an upper bound depending solely on the \emph{number of integer points} contained in the interior of the Newton polygon of the given real polynomial \cites{Kho1978,Mik2000}, which in turn is related to the monomials of the polynomial. Moreover, it has also been shown by Mikhalkin \cite{Mik2000} that this upper bound is also related to the connected components of the complement of the amoeba associated to the polynomial. For more details we refer to the survey of Mikhalkin~\cite{MikSurvey} and the book of Itenberg et al~\cite{Book}. For applications of such techniques of algebraic geometry to polynomial vector fields, we refer to Itenberg and Shustin~\cite{IteShu2000}. For applications of the relation between a polynomial vector fields and its Newton polygon, we refer to Dalbelo et al \cite{DalOliPer2024} and the references therein.

Sprott \cite{Sprott} brought also applied the idea of looking to the number of monomials to the field of qualitative theory of ordinary differential equations by seeking for the \emph{simplest} polynomial vector field in $\mathbb{R}^3$ exhibiting chaos. By \emph{simple} Sprott means with as few monomials as possible. In his own words: \emph{the simplicity refers to the algebraic representation rather than to the physical process described by the equations.} In particular, Sprott was able to find nineteen different quadratic vector fields defined on $\mathbb{R}^3$ exhibiting chaos and with either five monomials being two of them nonlinear, or six monomials being one of them nonlinear. Nowadays such quadratic vector fields are known as \emph{Sprott A, Sprott B,$\dots$, Sprott S}. For a qualitative study on some Sprott systems, we refer to \cite{MotOli2021} and references therein. Later Sprott \cite{Sprott2} was able to find a simpler chaotic system, with five monomials being only one nonlinear. From this point of view it is interesting to observe that the celebrated Lorenz \cite{Lorenz} and Rössler \cite{Rossler} systems are also quadratic, have seven monomials and, respectively, two or one of them are nonlinear. 

Following this notion of simple vector field, Gasull \cite{Gas2021} asks in his $8th$ problem for the minimal $m_0\in\mathbb{N}$ such that $\mathcal{H}^M(m_0)>m_0$, i.e. for the \emph{simplest vector field with more limit cycles than monomials}. On that time it was known that $4\leqslant m_0\leqslant9$ due to the cubic vector field of Li et al \cite{LiLiuYang2009}, with $9$ monomials and $13$ limit cycles. From Theorem~\ref{Main2} it now follows that $m_0=4$. As we will see in the proof that $\mathcal{H}^M(4)\geqslant 12,$  a system proving that $m_0=4$ is one of the form~\eqref{example1}, but the approach used in the proof only shows the existence of an example and it does not provide neither explicit exponents nor explicit parameters. On the other hand, a very simple explicit example showing that $\mathcal{H}^M(4)\geqslant4$ is 
\begin{equation}\label{example3}
	\dot x=ax^2y^5-ay, \quad \dot y=x^3y^2-x,
\end{equation}
with $a=-(1+\varepsilon)$ and $\varepsilon>0$ small enough. It has a limit cycle surrounding each one of the four critical points $(\pm1,\pm1)$ born via an Andronov-Hopf bifurcation, see the end of the proof of Theorem~\ref{Main2}.

While preparing a first version of this paper we thought that the first wanderings about the question of relating the number of limit cycles with the number of monomials were introduced in 2021 paper~\cite{BuzGasYag2021}, but this is not true. To the best of our knowledge the first authors to address this type of questions were Boros, Hofbauer and coauthors, see the 2019 papers~\cite{BorHofDCDS,BorHofJDE}. In fact, in a recent 2024 meeting they comment this fact to the first author and also that their approach could be used to get good lower bounds for $\mathcal{H}^M(4).$ We thank very much them for their suggestion that have leaded us to improve the lower bounds of a previous version of Theorem~\ref{Main2}.

The approach of counting the monomials of a vector field instead of its degree can be seen both as a strength or a weakness. This is so, because for instance affine changes of variables change  the number of monomials, but keep the degree. It is a weakness, because in most cases the number of monomials increases but it is a strength because occasionally it can go down. A similar situation happens with the degree by using birational transformations, together with time reparametrizations. In any case, it is an interesting point of view to try to go inside the study of the number of limit cycles of natural families of vector fields. 

Applications of this approach can be seen in the field of \emph{Chemical reaction network} (CRN) \cite{CRNT}, specially under the hypothesis of \emph{mass action kinetics} (MAK) \cite{150}. Roughly speaking, CRN models the behavior of real-world chemical systems, while MAK is the assumption that the \emph{rate of a chemical reaction is directly proportional to the product of the activities or concentrations of the reactants}. For example, this means that given a chemical reaction $A+2B\to C$, the \emph{rate of occurrence} of the reaction is given by $r(c_A,c_B)=\alpha c_Ac_B^2$, where $c_A$ and $c_B$ are the concentrations of the chemicals $A$ and $B$ and $\alpha\in\mathbb{R}$ is a constant. Therefore, given a system of interrelated chemical reactions, its dynamics is molded by a polynomial system of differential equations in which each monomial represents a reaction. For an introduction of the topic, we refer to Müller and Regensburger \cite{MulReg2014}. For other applications we refer to \cite[Chapter $7$]{ErdTot}.

To fix a simple example with a limit cycle when $m=4,$ we recall the  Higgins-Seklov model of glicolysis, see~\cite{Sel}. In adimensional form it writes as
\begin{equation*}
		\dot x=1-xy^2,\quad \dot y=axy^2-ay,
\end{equation*}
where $a$ is a real positive parameter.

This paper is organized as follows. In next section we include some preliminaries about the well-known Poin\-car\'e–Pontryagin Theorem. Theorems~\ref{Main1} and \ref{Main2} are proved in Section~\ref{Sec3}. The work ends with a small section with further thoughts.
	
\section{The Poincar\'e–Pontryagin Theorem}\label{Sec2}

Given a polynomial (resp. analytic or smooth) function $H\colon\mathbb{R}^2\to\mathbb{R}$, we associate the planar polynomial (resp. analytic or smooth) vector field $X=(P,Q)$ given by
	\[P(x,y)=-\frac{\partial H}{\partial y}(x,y), \quad Q(x,y)=\frac{\partial H}{\partial x}(x,y).\]
In this case we say that $X$ is \emph{Hamiltonian} and that $H$ is its Hamiltonian function. In particular, observe that $H$ is a first integral of $X$. Suppose that $X$ has a continuum of periodic orbits 
	\[A=\{\gamma_h\colon h\in(a,b)\}\subset\{(x,y)\in\mathbb{R}^2\colon H(x,y)\in(a,b)\},\]
with $\gamma_h$ depending continuously on $h$. See Figure~\ref{Fig1}(a). A maximal set with this property is called a {\it period annulus}.

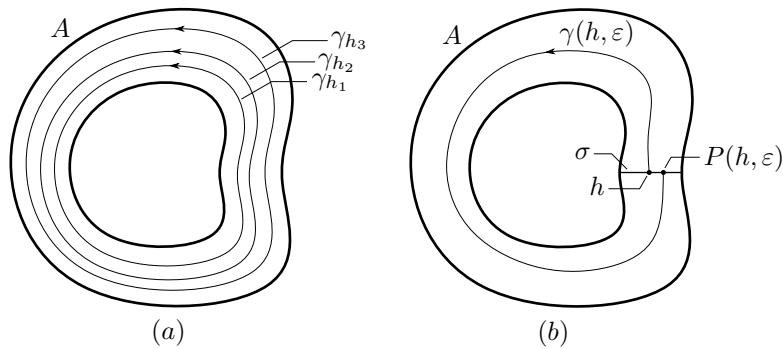
\begin{figure}[h]
	\begin{center}
		\begin{minipage}{5cm}
			\begin{center}
				\begin{overpic}[height=4cm]{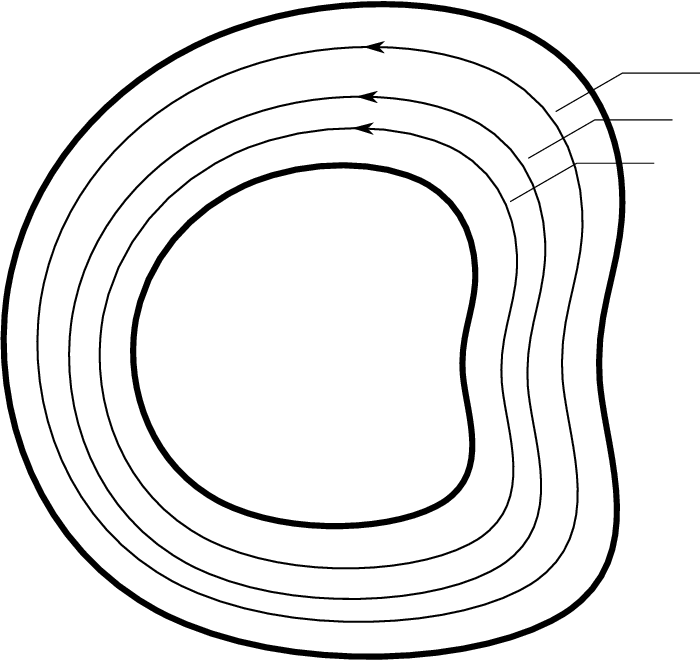} 
					\put(101,83){$\gamma_{h_3}$}
					\put(97,76){$\gamma_{h_2}$}
					\put(95,70){$\gamma_{h_1}$}
					\put(13,85){$A$}
				\end{overpic}
				
				$(a)$
			\end{center}
		\end{minipage}
		\begin{minipage}{5cm}
			\begin{center}
				\begin{overpic}[height=4cm]{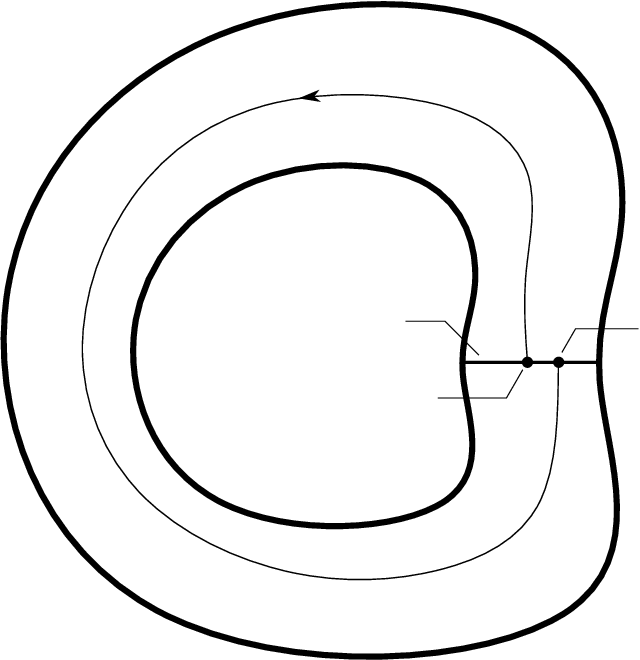} 
						\put(11,88){$A$}
						\put(60,37.5){$h$}
						\put(55,49.5){$\sigma$}
						\put(98,47.5){$P(h,\varepsilon)$}
						\put(50,89){$\gamma(h,\varepsilon)$}
				\end{overpic}
				
				$(b)$
			\end{center}
		\end{minipage}
	\end{center}
	\caption{Illustration of a continuum of periodic orbits and a displacement map.}\label{Fig1}
\end{figure}

Let $X_\varepsilon=(P_\varepsilon,Q_\varepsilon)$ be a perturbation of $X$ given by
	\[P_\varepsilon(x,y)=P(x,y)+\varepsilon f(x,y), \quad Q_\varepsilon(x,y)=Q(x,y)+\varepsilon g(x,y),\]
with $f$, $g\colon\mathbb{R}^2\to\mathbb{R}$ real polynomials and $|\varepsilon|$ small. Let $\sigma\subset A$ be a segment that is transversal to every periodic orbit $\gamma_h\subset A$ of the unperturbed vector field $X$. Given $h\in(a,b)$ and $\varepsilon\neq0$ small, let $\gamma(h,\varepsilon)$ be the piece of orbit of the perturbed vector field $X_\varepsilon$ between the starting point $h$ on $\sigma$ and the next intersection point $P(h,\varepsilon)$ with $\sigma$. See Figure~\ref{Fig1}(b). Let $d(h,\varepsilon)=P(h,\varepsilon)-h$ be the \emph{displacement map} associated to the perturbation $X_\varepsilon$. As usual, observe that $\gamma(h,\varepsilon)$ is a periodic orbit of $X_{\varepsilon}$ (resp. limit cycle) if, and only if, $(h,\varepsilon)$ is a zero (resp. isolated zero) of the displacement map. Moreover, given $h\in(a,b)$ we associate to $\gamma_h$ the line integral
\begin{equation}\label{1}
	I(h)=\oint_{\gamma_h}f\;dy-g\;dx,
\end{equation}
known as \emph{Abelian Integral}. 

\begin{theorem}[Poincaré–Pontryagin]\label{PP}
	Let $X_\varepsilon$, $d(h,\varepsilon)$ and $I(h)$ be as above. Then
	\begin{equation}\label{2}
		d(h,\varepsilon)=\varepsilon I(h)+\varepsilon^2\varphi(h,\varepsilon),
	\end{equation}
	where $\varphi(h,\varepsilon)$ is analytic and uniformly bounded for $(h,\varepsilon)$ in a neighborhood of $(h,0)$, $h\in(a,b)$.
\end{theorem}

For a proof of Theorem~\ref{PP}, see Christopher et al \cite[p. $143$]{ChrLiTor2024}. It follows from \eqref{2} that if $I$ is well defined on $(h_1,h_2)$ and $I(h_1)I(h_2)<0$, then for $|\varepsilon|>0$ small enough $\gamma(h_1,\varepsilon)$ and $\gamma(h_2,\varepsilon)$ bound, together with two segments of $\sigma$, a positive or negative invariant region of $X_\varepsilon$. See Figure~\ref{Fig2}. 
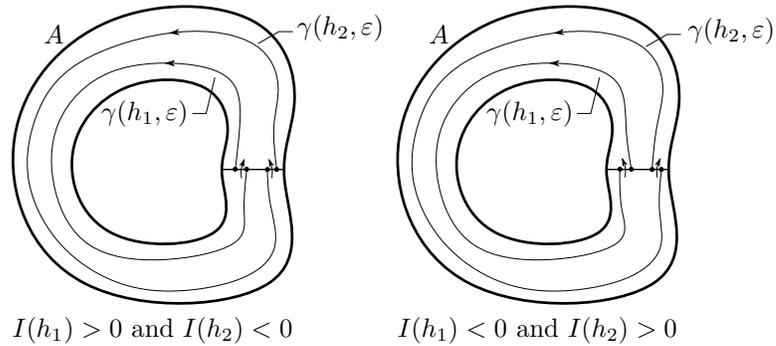
\begin{figure}[h]
	\begin{center}
		\begin{minipage}{5cm}
			\begin{center}
				\begin{overpic}[height=4cm]{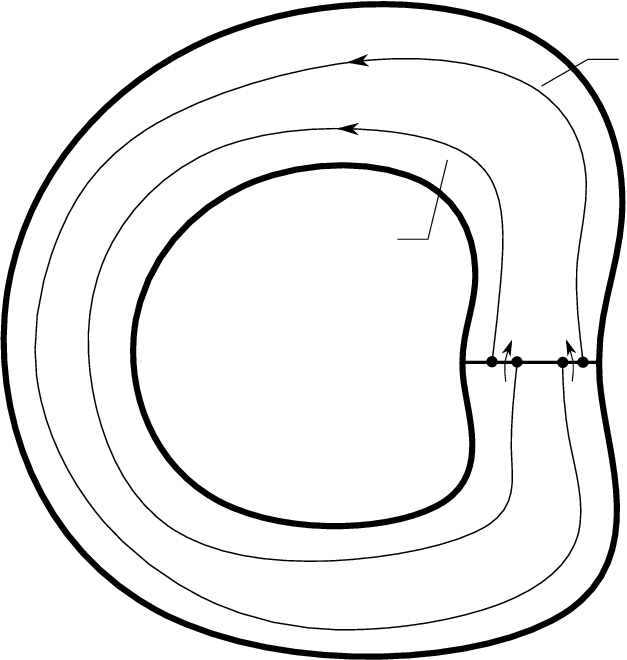} 
					\put(30,61.5){$\gamma(h_1,\varepsilon)$}
					\put(95,90){$\gamma(h_2,\varepsilon)$}
					\put(11,87){$A$}
				\end{overpic}
				
				$I(h_1)>0$ and $I(h_2)<0$
			\end{center}
		\end{minipage}
		\begin{minipage}{5cm}
			\begin{center}
				\begin{overpic}[height=4cm]{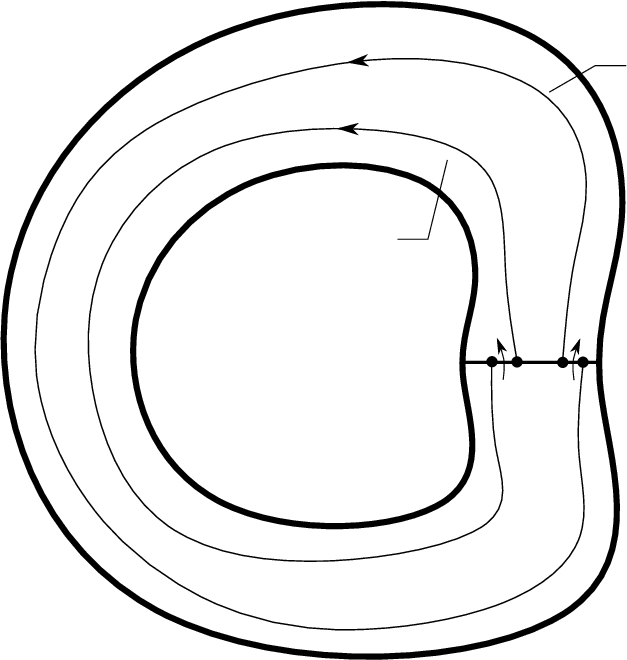} 
					\put(30,61.5){$\gamma(h_1,\varepsilon)$}
					\put(97,89){$\gamma(h_2,\varepsilon)$}
					\put(11,87){$A$}
				\end{overpic}
				
				$I(h_1)<0$ and $I(h_2)>0$
			\end{center}
		\end{minipage}
	\end{center}
	\caption{Illustration of the positive and negative invariant regions.}\label{Fig2}
\end{figure}
Hence, it follows from the Poincar\'e-Bendixson Theorem that $X_\varepsilon$ has at least one limit cycle between $\gamma(h_1,\varepsilon)$ and $\gamma(h_2,\varepsilon)$. Therefore, we have the following well-known corollary.

\begin{corollary}\label{CoroPP}
	Let $X_\varepsilon$ and $I(h)$ be as above. If $I$ is well defined on $(h_1,h_2)$ and $I(h_1)I(h_2)<0$, then there is $\varepsilon_0>0$ such that $X_\varepsilon$ has at least one limit cycle between $\gamma_{h_1}$ and $\gamma_{h_2}$, for $0<|\varepsilon|<\varepsilon_0$.
\end{corollary}

\begin{remark}\label{RemarkPP}
	Let $I(h)$ be as in \eqref{1}. It follows from Green's Theorem that if $\gamma_h$ is positively oriented then
		\[I(h)=\iint_{\Gamma_h}\frac{\partial f}{\partial x}+\frac{\partial g}{\partial y}\;dxdy,\]
	where $\Gamma_{h}\subset\mathbb{R}^2$ is the interior region bounded by $\gamma_h$.
\end{remark}

\section{Proof of the main results}\label{Sec3}

\begin{lemma}\label{Lemma1}
	Given $r\in\mathbb{Z}_{\geqslant0}$, let $X_r=(P,Q_r)$ be the planar polynomial vector field given by,
	\begin{equation}\label{3}
		P(x,y)=P(y)=y-y^3, \quad Q_r(x,y)=Q_r(x)=x\prod_{k=-r}^{r}(x-k).
	\end{equation}
	Then the following statements hold.
	\begin{enumerate}[label=(\roman*)]
		\item $X_r$ is Hamiltonian
		\item $X_r$ has $r+3$ monomials
		\item $X_r$ has $r+1$ centers on each of the lines $y=\pm1$ and $r$ centers on the line $y=0.$
	\end{enumerate}
\end{lemma}

\begin{proof} Statements $(i)$ and $(ii)$ follow directly from \eqref{3}. Hence, we focus on statement $(iii)$. Observe that the singularities of $X_r$ on the lines $y=\pm1$ are given by $(j,\pm1)$, with $j\in\{-r,\dots,r\}$. The Jacobian matrix at these singularities is given by,
	\[DX(j,\pm1)=\left(\begin{array}{cc} 0 & -2 \vspace{0.2cm} \\ Q_r'(j) & 0 \end{array}\right).\]
Hence,
\begin{equation}\label{4}
	\det DX(j,\pm1)=2Q_r'(j)=2(-1)^{r-j}\prod_{\substack{k=-r\\ k\neq j}}^r|j-k|.
\end{equation}
Since $X_r$ is Hamiltonian, it follows from \eqref{4} that $(j,\pm1)$ is either a hyperbolic saddle or a center, with the later occurring if, and only if, $\det DX(j,\pm1)>0$. Thus, we get from \eqref{4} that $(j,\pm1)$ is a center if, and only if, $j\equiv r\mod 2$. Therefore, either with $r$ even or odd, it is easy to see that we have exactly $r+1$ centers in each of the lines $y=\pm1$.  The study of the critical points on the line $y=0$ is similar.\end{proof}

\begin{proposition}\label{Lemma2}
	Given $r\in\mathbb{Z}_{\geqslant0}$, let $P(y)$ and $Q_r(x)$ be given by \eqref{3}. Then given $n\geqslant1$, there is a polynomial $R_n\colon\mathbb{R}^2\to\mathbb{R}$ with $n+1$ monomials and $\varepsilon_0>0$ such that the perturbed system $X_{n,r}=(P_{n},Q_{r})$ given by
		\[P_{n}(x,y)=P(y)+\varepsilon R_n(x,y), \quad Q_{r}(x,y)=Q_r(x),\]
	has at least 
		\[2n(r+1)+n\big(1+(-1)^r\big)\]
	limit cycles, for $0<|\varepsilon|<\varepsilon_0$. In particular, $X_{n,r}$ has $n+r+4$ monomials.
\end{proposition}

\begin{proof} Let $p_k=(x_k,-1)$, $k\in\{1,\dots,r+1\}$, and $p_k=(x_k,1)$, $k\in\{r+2,\dots,2r+2\}$, be the centers of $X_r$ such that $x_i<x_j$ for $i<j\leqslant r+1$ and $x_i>x_j$ for $i>j\geqslant r+2$. See Figure~\ref{Fig3}.
\begin{figure}[h]
	\begin{center}
		\begin{overpic}[width=10cm]{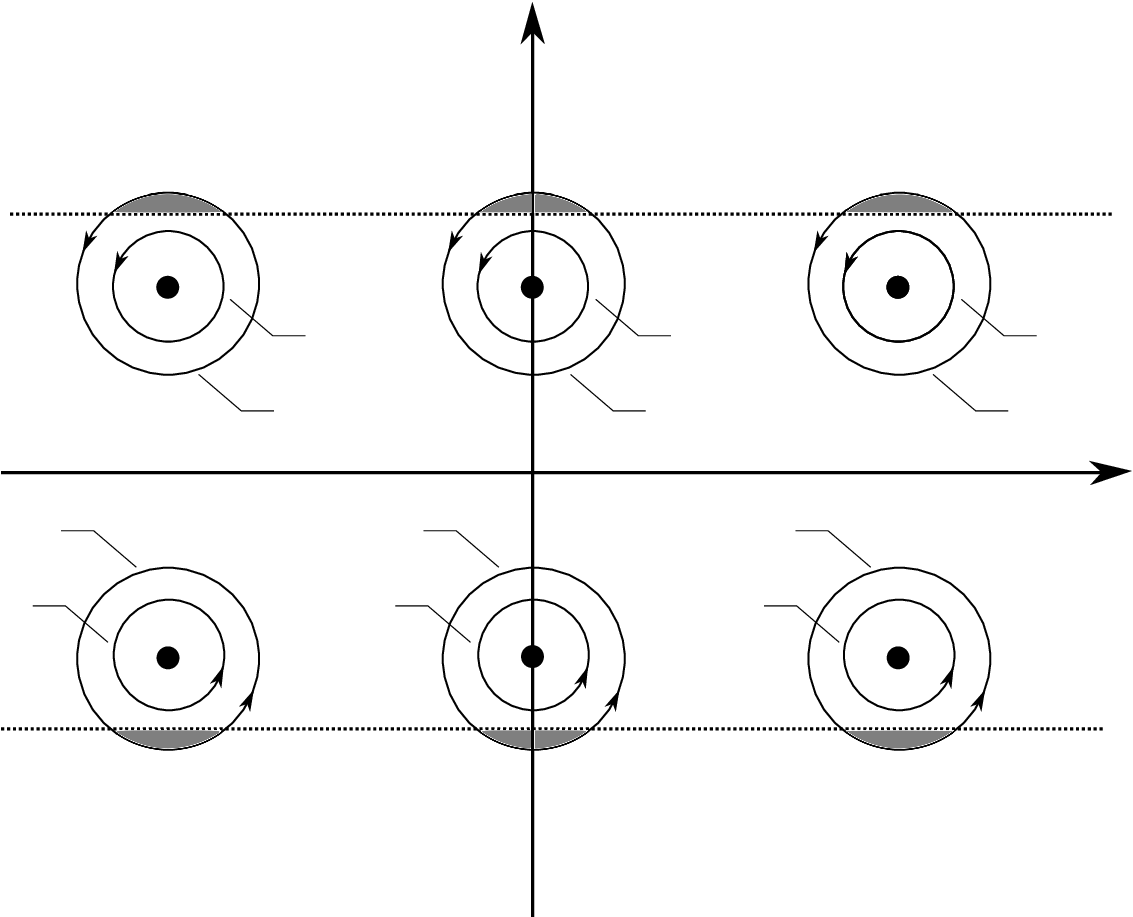} 
			\put(99,40.5){$x$}
			\put(49,79){$y$}
			\put(99,61.5){$y=a_1$}
			\put(99,16){$y=-a_1$}
			
			\put(14,25){$p_1$}
			\put(-1,27){$\gamma_0^1$}
			\put(2,33.5){$\gamma_1^1$}
			
			\put(48.5,23){$p_2$}
			\put(31,27){$\gamma_0^2$}
			\put(33.5,33.5){$\gamma_1^2$}
			
			\put(78,25){$p_3$}
			\put(64,27){$\gamma_0^3$}
			\put(66.5,33.5){$\gamma_1^3$}
			
			\put(78,52.5){$p_4$}
			\put(92,50.5){$\gamma_0^4$}
			\put(89.5,44){$\gamma_1^4$}
			
			\put(48.5,55){$p_5$}
			\put(60,50.5){$\gamma_0^5$}
			\put(58,44){$\gamma_1^5$}
			
			\put(14,52.5){$p_6$}
			\put(28,50.5){$\gamma_0^6$}
			\put(25,44){$\gamma_1^6$}
		\end{overpic}
	\end{center}
\caption{Illustration of $p_k$ and $\gamma_i^k$, for $r=2$ and $n=1$.}\label{Fig3}
\end{figure}
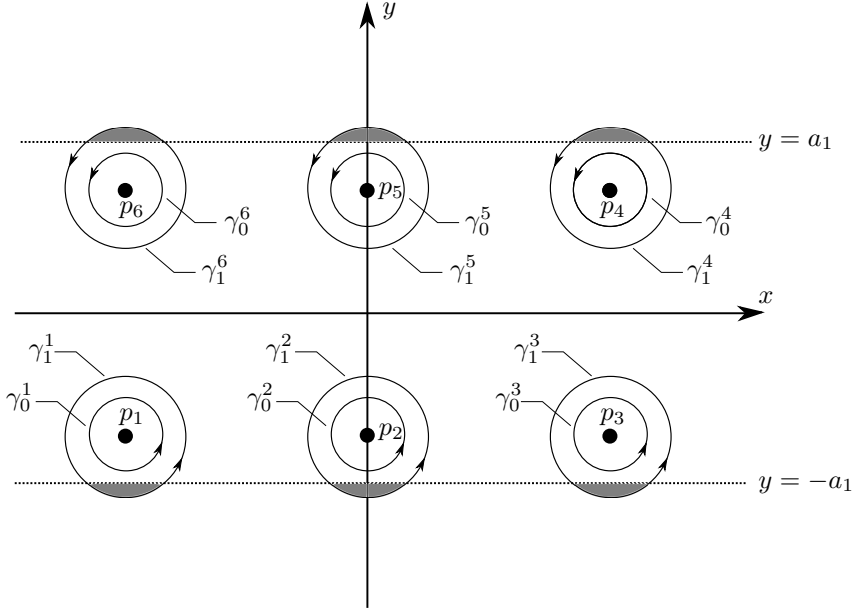
Let $A_k$ be the period annulus associated to $p_k$ and let $\gamma_0^k,\gamma_1^k,\dots,\gamma_n^k$ be fixed periodic orbits in $A_k$, $k\in\{1,\dots,2r+2\}$, such that $\gamma_{i-1}^k\subset\Gamma_i^k$, where $\Gamma_i^k$ is the open interior region bounded by $\gamma_i^k$, $i\in\{1,\dots,n\}$. See Figure~\ref{Fig3}. Observe that each $\gamma_i^k$ is positively oriented, $i\in\{0,\dots,n\}$, $k\in\{1,\dots,2r+2\}$. Let
	\[\alpha_i^k=\sup\{|y|\colon (x,y)\in\Gamma_i^k\},\]
$i\in\{0,\dots,n\}$, $k\in\{1,\dots,2r+2\}$. Observe that $\alpha_n^k>\dots>\alpha_0^k>0$, for each $k\in\{1,\dots,2r+2\}$. Observe also that we can choose $\gamma_0^k,\gamma_1^k,\dots,\gamma_n^k$ such that for each $i\in\{1,\dots,n\}$ there is $a_i>0$ satisfying $\alpha_{i-1}^k<a_i<\alpha_i^k$, $k\in\{1,\dots,2r+2\}$. See Figure~\ref{Fig3}. Given a polynomial $R\colon\mathbb{R}^2\to\mathbb{R}$ and a periodic orbit $\gamma$ of $X_r$, set
	\[I(R,\gamma)=\iint_{\Gamma}\frac{\partial R}{\partial x}(x,y)\;dxdy,\]
where $\Gamma$ is the interior region bounded by $\gamma$. It follows from Remark~\ref{RemarkPP} that if $\gamma$ is positively oriented, then $I(R,\gamma)$ is the Abelian integral of the perturbed vector field 
	\[P_{n}(x,y)=P(y)+\varepsilon R(x,y), \quad Q_{r}(x,y)=Q_r(x),\]
associated to $\gamma$. Let $R_0(x)=x^{2n+1}$ and observe that $I(R_0,\gamma_i^k)>0$ for every $i\in\{0,\dots,n\}$ and $k\in\{1,\dots,2r+2\}$. Given $m_1\geqslant1$ let,
	\[R_1(x,y)=R_1(x,y;m_1)=x^{2n+1}-x^{2n-1}\left(\frac{y}{a_1}\right)^{2m_1}.\]
We claim that there is $m_1\geqslant1$ big enough such that $I(R_1,\gamma_0^k)>0$ and $I(R_1,\gamma_1^k)<0$, for every $k\in\{1,\dots,2r+2\}$. Indeed, first observe that if $y\in\mathbb{R}$ is such that $|y|<a_1$, then
\begin{equation}\label{5}
	\lim\limits_{m_1\to\infty}\left(\frac{y}{a_1}\right)^{2m_1}=0.
\end{equation}
Hence, it follows from \eqref{5}, from the compactness of $\overline{\Gamma_0^k}$ (i.e. the topological closure of $\Gamma_0^k$) and from the fact that $\alpha_0^k<a_1$ that
\begin{equation}\label{6}
	\lim\limits_{m_1\to\infty}x^{2n-2}\left(\frac{y}{a_1}\right)^{2m_1}=0,
\end{equation}
uniformly in $(x,y)\in\Gamma_0^k$, $k\in\{1,\dots,2r+2\}$. Thus we have,
\[\begin{array}{rl}
	\displaystyle \lim\limits_{m_1\to\infty} I(R_1,\gamma_0^k) &\displaystyle= \lim\limits_{m_1\to\infty}\iint_{\Gamma_0^k}(2n+1)x^{2n}-(2n-1)x^{2n-2}\left(\frac{y}{a_1}\right)^{2m_1}\;dxdy \vspace{0.2cm} \\
	&\displaystyle= I(R_0,\gamma_0^k)-\lim\limits_{m_1\to\infty}\iint_{\Gamma_0^k}(2n-1)x^{2n-2}\left(\frac{y}{a_1}\right)^{2m_1}\;dxdy \vspace{0.2cm} \\
	&\displaystyle= I(R_0,\gamma_0^k)>0,
\end{array}\]
for $k\in\{1,\dots,2r+2\}$, with the last equality following from \eqref{6}. Let
	\[\Omega_i^k=\{(x,y)\in\Gamma_i\colon |y|>a_i\},\]
and observe that $\Omega_i^k$ has positive Lebesgue measure, $i\in\{1,\dots,n\}$. See the gray-shaded area in Figure~\ref{Fig3}. Hence, it follows that
	\[\lim\limits_{m_1\to\infty}\iint_{\Gamma_1^k}x^{2n-2}\left(\frac{y}{a_1}\right)^{2m_1}\;dxdy\geqslant\lim\limits_{m_1\to\infty}\iint_{\Omega_1^k}x^{2n-2}\left(\frac{y}{a_1}\right)^{2m_1}\;dxdy =+\infty.\]
Therefore,
\[\begin{array}{rl}
	\displaystyle \lim\limits_{m_1\to\infty} I(R_1,\gamma_1^k) &\displaystyle= \lim\limits_{m_1\to\infty}\iint_{\Gamma_1^k}(2n+1)x^{2n}-(2n-1)x^{2n-2}\left(\frac{y}{a_1}\right)^{2m_1}\;dxdy \vspace{0.2cm} \\
	&\displaystyle\leqslant I(R_0,\gamma_1^k)-\lim\limits_{m_1\to\infty}\iint_{\Omega_1^k}(2n-1)x^{2n-2}\left(\frac{y}{a_1}\right)^{2m_1}\;dxdy \vspace{0.2cm} \\
	&\displaystyle=-\infty.
\end{array}\]
This proves the claim. That is, there is $m_1\geqslant1$ big enough such that
\begin{equation}\label{7}
	I(R_1,\gamma_0^k)>0, \quad I(R_1,\gamma_1^k)<0,
\end{equation}
for every $k\in\{1,\dots,2r+2\}$. From now on, we fix $m_1\in\mathbb{N}$ big enough such that \eqref{7} is satisfied. It follows from the proof of Lemma~\ref{Lemma1} that if $r$ is even, then 
	\[p_{k_1}=p_{r/2+1}, \quad p_{k_2}=p_{{3r}/{2}+2}\]
lie on the line $x=0$. See Figure~\ref{Fig3}. We claim that we can choose $\gamma_{-1}^{k_j}\subset\Gamma_0^{k_j}$ such that $I(R_1,\gamma_{-1}^{k_j})<0$, $j\in\{1,2\}$. Indeed, let
\begin{equation}\label{8}
	b_0^{k_j}=\inf\{|y|\colon(x,y)\in\Gamma_0^{k_j}\},
\end{equation}
and observe that $b_0^{k_j}>0$, $j\in\{1,2\}$. Observe also that
\begin{equation}\label{9}
	\frac{\partial R_1}{\partial x}(x,y)<0 \Leftrightarrow x^2<\frac{2n-1}{2n+1}\left(\frac{y}{a_1}\right)^{2m_1}.
\end{equation}
Let $\gamma_{-1}^{k_j}\subset\Gamma_0^{k_j}$ be of small enough amplitude such that
\begin{equation}\label{10}
	(x,y)\in\Gamma_{-1}^{k_j}\Rightarrow x^2<\frac{2n-1}{2n+1}\left(\frac{b_0^{k_j}}{a_1}\right)^{2m_1},
\end{equation}
where $\Gamma_{-1}^{k_j}$ is the interior region bounded by $\gamma_{-1}^{k_j}$, $j\in\{1,2\}$. Observe that it is possible to choose $\gamma_{-1}^{k_j}$ precisely because $p_{k_j}$ lies in the line $x=0$ and it is not the origin, $j\in\{1,2\}$. Hence, it follows from \eqref{8}, \eqref{9} and \eqref{10} that
	\[\left.\frac{\partial R_1}{\partial x}(x,y)\right|_{\Gamma_{-1}^{k_j}}<0,\]
and thus we have $I(R_1,\gamma_{-1}^{k_j})<0$, $j\in\{1,2\}$. This proves the claim. Therefore, it follows that if $|\varepsilon|>0$ is small enough, then the perturbed vector field $X_{1,r}=(P_{1},Q_{r})$ given by
	\[P_{1}(x,y)=P(y)+\varepsilon R_1(x,y), \quad Q_{r}(x,y)=Q_r(x),\]
has $r+5$ monomials and at least $2(r+1)+1+(-1)^r$ limit cycles, being $2(r+1)$ of them bifurcating between the orbits $\gamma_0^k$ and $\gamma_1^k$, $k\in\{1,\dots,2r+2\}$ and the other (possibly) two between $\gamma_0^{k_j}$ and $\gamma_{-1}^{k_j}$, $j\in\{1,2\}$, when $r$ is even. Similarly, we can continue this process and obtain moreover another family of $2(r+1)+1+(-1)^r$ cycles by considering,
	\[R_2(x,y)=R_2(x,y;m_1,m_2)=x^{2n+1}-x^{2n-1}\left(\frac{y}{a_1}\right)^{2m_1}+x^{2n-3}\left(\frac{y}{a_2}\right)^{2m_2}.\] Then, for this vector field we have obtained $4(r+1)+2(1+(-1)^r)$ limit cycles.
More precisely, once obtained $R_1$, we can take $m_2>m_1$ big enough such that none of the previous Abelian integrals changes sign at the same time that $I(R_2,\gamma_2^k)>0$, $k\in\{1,\dots,2r+2\}$. Then, if $r$ is even, we can choose $\gamma_{-2}^{k_j}\subset\Gamma_{-1}^{k_j}$ small enough such that $I(R_2,\gamma_{-2}^{k_j})>0$, $j\in\{1,2\}$. 

Continuing this process, we obtain a perturbation of the form
	\[R_n(x,y)=\sum_{k=0}^{n}(-1)^kx^{2(n-k)+1}\left(\frac{y}{a_k}\right)^{2m_k},\]
with $a_0=1$, $m_0=0$ and $m_k\gg m_{k-1}$, $k\in\{1,\dots,n\}$, such that the perturbed vector field $X_{n,r}=(P_{n},Q_{r})$ given by
	\[P_{n}(x,y)=P(y)+\varepsilon R_n(x,y), \quad Q_{r}(x,y)=Q_r(x),\]
has $n+r+4$ monomials and at least $2n(r+1)+n\big(1+(-1)^r\big)$
limit cycles, for $|\varepsilon|>0$ small enough. \end{proof}

\begin{proof}[Proof of Theorem~\ref{Main1}] It follows from Proposition~\ref{Lemma2} that we have a two-parameter family of planar polynomial vector fields $X_{n,r}$, with $r\geqslant0$ and $n\geqslant1$, such that
\begin{equation}\label{11}
	\mathcal{H}^M(n+r+4)\geqslant 2n(r+1)+n\big(1+(-1)^r\big)\geqslant 2n(r+1).
\end{equation}
If we replace $m=n+r+4$ at \eqref{11} we obtain,
\begin{equation}\label{12}
	\mathcal{H}^M(m)\geqslant2(m-r-4)(r+1).
\end{equation}
In order to maximize the leading coefficient of the right-hand side of \eqref{12}, and knowing that $r$ must be an integer, we take
\begin{equation}\label{13}
	r=\frac{1}{2}m+\frac{(-1)^m-1}{4}.
\end{equation}
Replacing \eqref{13} at \eqref{12} we obtain,
\begin{equation}\label{14}
	\mathcal{H}^M(m)\geqslant\frac{1}{2}m^2-3m-8+\frac 94(1-(-1)^m)\geqslant\frac{1}{2}m^2-3m-8.
\end{equation}
This finishes the proof.\end{proof}
 
 \begin{proof}[Proof of Theorem~\ref{Main2}]
 Let $X_{n,r}$ be given by Proposition~\ref{Lemma2}. We recall that $X_{n,r}$ has $n+r+4$ monomials and at least
$2n(r+1)+n\big(1+(-1)^r\big)$
limit cycles, for $|\varepsilon|>0$ small. If we take $r=2$ and $n=3$ (resp. $n=4$) we obtain $\mathcal{H}^M(9)\geqslant24$ (resp. $\mathcal{H}^M(10)\geqslant32$). 

We now focus on the claim that $\mathcal{H}^M(m)\geqslant12$ for $m\in\{4,5,6\}$. Consider the analytic system defined on the open first quadrant of $\mathbb{R}^2$ and given by
\begin{equation}\label{15}
	\dot x= \alpha_1 x^{g_{11}}y^{g_{12}}-\beta_1 x^{h_{11}}y^{h_{12}}, \quad \dot y= \alpha_2 x^{g_{21}}y^{g_{22}}-\beta_2 x^{h_{21}}y^{h_{22}},
\end{equation}
with $\alpha_i$, $\beta_i$, $g_{ij}$, $h_{ij}\in\mathbb{R}$. It follows from Boros and Hofbauer \cite[Section $7$]{BorHofJDE} that for some choice of the parameters and exponents, system \eqref{15} has at least three limit cycles of odd multiplicity. In particular, such limit cycles persist under small perturbations. Therefore, we can take a rational approximation of such exponents and thus suppose that system \eqref{15} can be written as
\begin{equation}\label{16}
	\dot x= \alpha_1 x^{\frac{a_1}{b_1}}y^{\frac{c_1}{d_1}}-\beta_1 x^{\frac{a_2}{b_2}}y^{\frac{c_2}{d_2}}, \quad \dot y= \alpha_2 x^{\frac{a_3}{b_3}}y^{\frac{c_3}{d_3}}-\beta_2 x^{\frac{a_4}{b_4}}y^{\frac{c_4}{d_4}},
\end{equation}
with $a_i$, $c_i\in\mathbb{Z}$ and $b_i$, $d_i\in\mathbb{Z}_{>0}$ relatively primes and has yet at least three limit cycles of odd multiplicity. Let $b=2b_1b_2b_3b_4$, $d=2d_1d_2d_3d_4$ and observe that $b\geqslant 2$ and $d\geqslant2$ are even natural numbers.  Applying the non-reversible transformation $(x,y)=(u^b,v^d)$ we obtain a new vector field given by
	\[\begin{array}{l}
		\displaystyle \dot u= \frac{1}{bu^{b-1}}\left(\alpha_1u^{2a_1b_2b_3b_4}v^{2c_1d_2d_3d_4}-\beta_1 u^{2b_1a_2b_3b_4}v^{2d_1c_2d_3d_4}\right), \vspace{0.2cm} \\
		\displaystyle \dot v= \frac{1}{dv^{d-1}}\left(\alpha_2u^{2b_1b_2a_3b_4}v^{2d_1d_2c_3d_4}-\beta_2 u^{2b_1b_2b_3a_4}v^{2d_1d_2d_3c_4}\right).
	\end{array}\]
By using the rescaling of time  $dt/d\tau=bdu^{b-1+2k}v^{d-1+2k}$, with $k\in\mathbb{Z}_{>0}$, we obtain
\begin{equation}\label{17}
	\begin{array}{l}
		\displaystyle \dot u= dv^{d-1}\left(\alpha_1u^{2(a_1b_2b_3b_4+k)}v^{2(c_1d_2d_3d_4+k)}-\beta_1 u^{2(b_1a_2b_3b_4+k)}v^{2(d_1c_2d_3d_4+k)}\right), \vspace{0.2cm} \\
		\displaystyle \dot v= bu^{b-1}\left(\alpha_2u^{2(b_1b_2a_3b_4+k)}v^{2(d_1d_2c_3d_4+k)}-\beta_2 u^{2(b_1b_2b_3a_4+k)}v^{2(d_1d_2d_3c_4+k)}\right).
	\end{array}
\end{equation}
Observe that \eqref{17} is polynomial for $k\in\mathbb{Z}_{>0}$ big enough. Moreover, since $b\geqslant2$ and $d\geqslant2$ are even numbers, it follows that \eqref{17} is reversible in relation to the lines $u=0$ and $v=0$. Hence, \eqref{17} has diffeomorphic copies of \eqref{16} at each open quadrant and thus we obtain $\mathcal{H}^M(4)\geqslant12$. Since each of these limit cycles has odd multiplicity, it follows that they persist under small perturbations and thus we also have $\mathcal{H}^M(m)\geqslant12$ for $m\in\{5,6\}$. 

Finally, we now prove that $\mathcal{H}^M(8)\geqslant20$ and $\mathcal{H}^M(7)\geqslant16$. The proof will follow by studying the cyclicity of some weak foci. For a general theory of cyclicity of limit sets, we refer to Roussarie \cite[Chapter $2$]{Roussarie}. For more details about the cyclicity of weak focus in polynomial vector fields, we refer to Christopher et al \cite[Chapter $1$]{ChrLiTor2024}. For a more computational approach, we refer to Romanovski and Shafer \cite[Chapter $6$]{RomSha2009}. 
	
Consider the system with eight monomials 
\begin{equation}\label{18}
		\dot x=a_5y^6+a_4y^5+a_3y^4+a_2y^3+a_1xy^2-a, \quad
		\dot y=xy-1,
\end{equation}
where $a=a_1+\dots+a_5$.  It is not hard to see that if $a_j=a_j^*$, $j=1,\dots,5$, where
	\[a_1^*=-1, \quad a_2^*=-\frac{161}{17}, \quad a_3^*=\frac{17}{11}, \quad a_4^*=-\frac{6}{11}, \quad a_5^*=\frac{7}{99},\]
then the point $p=(1,1)$ is a weak focus of order five, i.e. it is not hyperbolic, $L_1=\dots=L_4=0$ and $L_5\neq0$, where $L_j$ is its $j$th \emph{Lyapunov constant} (see Adronov et al. \cite[p. 254]{Andronov}). Moreover, if we calculate the Jacobian matrix of $L_1,L_2,L_3,L_4$ in relation to $a_2,a_3,a_4,a_5$, at $a_j=a_j^*$, $j=2,3,4,5$, it can be seen that
	\[\det\frac{\partial (L_1,L_2,L_3,L_4)}{\partial(a_2,a_3,a_4,a_5)}(a_2^*,a_3^*,a_4^*,a_5^*)\neq0.\]
Hence, it follows from Christopher et al \cite[Theorem $1.5$]{ChrLiTor2024} that we can choose $a_j\approx a_j^*$, $j\in\{2,3,4,5\}$, such that four limit cycles bifurcate from $p$. Now we move $a_1$ to perturb the trace of \eqref{18} at $p$ and thus to bifurcate a fifth limit cycle (see Romanovski and Shafer \cite[Theorem $6.2.7$]{RomSha2009}). Therefore, for some specific values of the parameters, system \eqref{18} has at least five limit cycles near the point $p=(1,1)$ and surrounding it. Thus, similarly to the previous argumentation, we now use the non-invertible change of variables $(x,y)=(u^2,v^2)$, followed by the rescaling of time $dt/d\tau=2uv$, obtaining the new system
\begin{equation}\label{19}
		\dot u=a_5v^{13}+a_4v^{11}+a_3v^9+a_2v^7+a_1u^2v^5-av, \quad
		\dot v=u^3v^2-u.
\end{equation}
It has again eight monomials and moreover it has a diffeomorphic copy, in each open quadrant, of the first open quadrant of \eqref{18}. In particular, it has $20$ limit cycles for some values of the coefficients and thus $\mathcal{H}^M(8)\geqslant20$. To prove $\mathcal{H}^M(7)\geqslant16$, we substitute $a_5=0$ in \eqref{18}, obtaining a system with seven mononials. In this system, if $a_j=\overline{a_j}$, $j=1,\ldots,4$, where
	\[\overline{a_1}=-1, \quad \overline{a_2}=-\frac{42}{109}, \quad \overline{a_3}=\frac{31}{109}, \quad \overline{a_4}=-\frac{6}{109},\]
then $p=(1,1)$ is a weak focus of order four and the proof follows similarly.

For each $k=1,2,3$ by taking $a_5=a_4=..=a_{k+1}=0$  and suitable $a_1,\ldots, a_{k}$ in \eqref{19} we obtain a vector field with $k+3$ monomials and at least $4k$ limit cycles, with $k$ of them included in each quadrant. These results give less limit cycles that the examples constructed from the  Boros and coauthor's result when $m=4,5$ and by taking $k=3$ gives a different proof that $\mathcal{H}^M(6)\ge12,$ with the advantage that this new example is explicit.
\end{proof}

\section{Further Thoughts}

Regarding the recent developments in the field of algebraic geometry described in the introduction, it is worthy to ask for a variant of the Hilbert number as a function of the associated newton polygon of the polynomial vector field. In particular, as a functions of the number of integer points contained in its interior. Notably, in the case of a Hamiltonian vector field $X$ associated with a polynomial~$p$, the Newton polygons $N(X)$ and $N(p)$ coincide, differing only by a translation in $\mathbb{Z}^2$. This observation, combined with the discussion made in the introduction, could be used for instance to establish a bound on the number of distinct periodic annuli of $X$ in terms of the number of integer points in $N(X)$.  We thank very much the anonymous reviewers for pointing out such developments and suggesting this variation of the problem.

\section*{Acknowledgments}

We thank to the reviewers their comments and suggestions which help us to improve the presentation of this paper. This work is supported by the Spanish State Research Agency, through the projects PID2022-136613NB-I00 grant and the Severo Ochoa and Mar\'ia de Maeztu Program for Centers and Units of Excellence in R\&D (CEX2020-001084-M),  grant 2021-SGR-00113 from AGAUR, Generalitat de Ca\-ta\-lu\-nya, and by S\~ao Paulo Research Foundation (FAPESP), grants 2019/10269-3, 2021/01799-9 and 2022/14353-1.

\end{document}